\documentclass[11pt,a4paper]{amsart}
\usepackage{mathrsfs}
\usepackage{syntonly}
\usepackage{amsmath}
\usepackage{amsthm}
\usepackage{amsfonts}
\usepackage{amssymb}
\usepackage{latexsym}
\usepackage{amscd,amssymb,amsopn,amsmath,amsthm,graphics,amsfonts,mathrsfs,accents,enumerate,verbatim,calc}
\usepackage[dvips]{graphicx}
\usepackage[colorlinks=true,linkcolor=red,citecolor=blue]{hyperref}
\usepackage[all]{xy}
\usepackage{enumitem}

\date{}
\allowdisplaybreaks[4] \footskip=15pt
\renewcommand{\uppercasenonmath}[1]{}

\numberwithin{equation}{section} \theoremstyle{plain}
\newtheorem{lem}{Lemma}[section]

\newtheorem{prop}[lem]{Proposition}
\newtheorem{thm}[lem]{Theorem}

\newtheorem{definition}[lem]{Definition}
\newtheorem{Ex}[lem]{Example}
\newtheorem{Quest}[lem]{Question}
\newtheorem{Property}[lem]{Property}
\newtheorem{Properties}[lem]{Properties}
\newtheorem{Subprops}{}[lem]
\newtheorem{Para}[lem]{}

\newtheorem{remark}[lem]{Remark}
\newtheorem{rem}[lem]{Remark}

\newenvironment{para}{\begin{Para}\rm}{\end{Para}}

\newtheorem*{ack*}{ACKNOWLEDGEMENTS}

\newcommand{\pf}{\noindent\begin {proof}}
\newcommand{\epf}{\end{proof}}

\begin{document}
\begin{center}
{\Large  \bf  Balanced pairs on triangulated categories}

\vspace{0.5cm} Xianhui Fu$^a$, Jiangsheng Hu$^b$, Dongdong Zhang$^c$ and Haiyan Zhu$^d$\footnote{Corresponding author. Xianhui Fu was supported by the NSF of China (Grant No. 12071064). Jiangsheng Hu was supported by the NSF of China (Grant Nos. 12171206, 11771212) and  Qing Lan Project of Jiangsu Province. Haiyan Zhu was supported by the Natural Science Foundation of Zhejiang Provincial (LY18A010032).} \\
\medskip

\hspace{-4mm}
$^a$School of Mathematics and Statistics, Northeast Normal University, Changchun 130024, China\\
  $^b$Department of Mathematics, Jiangsu University of Technology,
 Changzhou 213001, China\\

 $^c$Department of Mathematics, Zhejiang Normal University,
 Jinhua 321004, China\\
 $^d$College of Science, Zhejiang University of Technology, Hangzhou 310023, China\\

E-mails: fuxianhui@gmail.com, jiangshenghu@jsut.edu.cn, zdd@zjnu.cn and hyzhu@zjut.edu.cn\\
\end{center}

\bigskip
\medskip
\centerline { \bf  Abstract}
\leftskip10truemm \rightskip10truemm \noindent
Let $\mathcal{C}$ be a triangulated category. We first introduce the notion of balanced pairs in $\mathcal{C}$, and then establish the bijective correspondence between balanced pairs and proper classes $\xi$ with enough $\xi$-projectives and enough $\xi$-injectives.
Assume that $\xi:=\xi_{\mathcal{X}}=\xi^{\mathcal{Y}}$ is the proper class  induced by a balanced pair $(\mathcal{X},\mathcal{Y})$. We prove that $(\mathcal{C}, \mathbb{E}_\xi, \mathfrak{s}_\xi)$ is an extriangulated category. Moreover, it is proved that $(\mathcal{C}, \mathbb{E}_\xi, \mathfrak{s}_\xi)$ is a triangulated category if and only if $\mathcal{X}=\mathcal{Y}=0$; and that $(\mathcal{C}, \mathbb{E}_\xi, \mathfrak{s}_\xi)$ is an exact category if and only if $\mathcal{X}=\mathcal{Y}=\mathcal{C}$. As an application, we produce a large variety of examples of extriangulated categories which are neither exact nor triangulated.
\\[2mm]

\leftskip0truemm \rightskip0truemm
\section {Introduction}

Relative homological algebra has been formulated by Hochschild in
categories of modules and later by Heller, Butler and Horrocks in more general categories with a relative abelian structure. Beligiannis developed in \cite{Bel1} a relative version of homological algebra in triangulated categories in analogy to relative homological algebra in abelian categories, in which the notion of a proper class of exact sequences is replaced by a proper class of triangles. By specifying a class of triangles $\xi$, which is called a proper class of triangles, he introduced the full subcategory $\mathcal{P}(\xi)$ of $\xi$-projective objects and the full subcategory $\mathcal{I}(\xi)$ of $\xi$-injective objects

The notion of balanced pairs involving two additive full subcategories
of an abelian category $\mathcal{A}$ was introduced by Chen \cite{chen}, which generalizes projectives and injectives from homological aspects. In
general, a balanced pair always shares many similar properties with the projectives and injectives. We refer to Chen \cite{chen} for more details. It should be noted that this subject first appeared in Enochs' work (see \cite{EJ} for instance), where a pair $(\mathcal{H}, \mathcal{G})$ in $\mathcal{A}$ is balanced if
and only if the bifunctor $\textrm{Hom}_{\mathcal{A}}(-,-)$ is right balanced by $\mathcal{H}\times\mathcal{G}$. For examples of balanced pairs, the reader may refer to \cite[Chapter 8]{EJ}. An interesting and deep
result in \cite{WLH} is that in an abelian category $\mathcal{A}$ with small Ext groups, there
exists a one-to-one correspondence between balanced pairs and Quillen exact structures $\xi$ with enough $\xi$-projectives and enough $\xi$-injectives. Motivated by this, it seems natural to introduce the notion of balanced pairs in a triangulated category $\mathcal{C}$, and to establish certain relations connecting balanced pairs with certain classes of triangles in $\mathcal{C}$. Thus, we have the following main result of this paper.

\begin{thm}\label{main-theorem}\label{thm:3.11} Let $\mathcal{C}$ be a triangulated category.
The assignments
\begin{center}
 $\Psi:(\mathcal{X},\mathcal{Y})\mapsto$~$\xi_{\mathcal{X}}=\xi^{\mathcal{Y}}$~$~~$ and $~~$~$\Phi:\xi\mapsto(\mathcal{P}(\xi),\mathcal{I}(\xi))$
\end{center}
  give mutually inverse bijections between the following classes:

\begin{enumerate}
\item Balanced pairs $(\mathcal{X},\mathcal{Y})$ in $\mathcal{C}$.

\item Proper classes $\xi$ in $\mathcal{C}$ with enough $\xi$-projectives and enough $\xi$-injectives.
\end{enumerate}
\end{thm}
We note that the precise definition of balanced pairs in a triangulated category $\mathcal{C}$ can be seen in Definition \ref{df:3.7} below. Some examples of balanced pairs are presented Examples \ref{Ex:3.8}-\ref{Ex:3.10}.

Assume that $\mathbb{E}: \mathcal{C}^{\rm op}\times \mathcal{C}\rightarrow {\rm Ab}$ is an additive bifunctor, where $\mathcal{C}$ is an additive category and ${\rm Ab}$ is the category of abelian groups. For any objects $A, C\in\mathcal{C}$, an element $\delta\in \mathbb{E}(C,A)$ is called an $\mathbb{E}$-extension.
Let $\mathfrak{s}$ be a correspondence which associates an equivalence class $$\mathfrak{s}(\delta)=\xymatrix@C=0.8cm{[A\ar[r]^x
 &B\ar[r]^y&C]}$$ to an $\mathbb{E}$-extension $\delta\in\mathbb{E}(C, A)$. This $\mathfrak{s}$ is called a {\it realization} of $\mathbb{E}$, if it makes the diagram in \cite[Definition 2.9]{NP} commutative.
 A triplet $(\mathcal{C}, \mathbb{E}, \mathfrak{s})$ is called an {\it extriangulated category} if it satisfies the following conditions:
\begin{enumerate}
\item $\mathbb{E}\colon\mathcal{C}^{\rm op}\times \mathcal{C}\rightarrow \rm{Ab}$ is an additive bifunctor;

\item $\mathfrak{s}$ is an additive realization of $\mathbb{E}$; and

\item $\mathbb{E}$ and $\mathfrak{s}$ satisfy certain axioms in \cite[Definition 2.12]{NP}.
\end{enumerate}

Exact categories and extension closed subcategories of an extriangulated category are extriangulated categories. In particular, every triangulated category is an extriangulated category. More precisely,
assume that $\mathcal{C}$ is a triangulated category with  suspension functor $\Sigma$. For any $\delta\in{\mathbb{E}(C,A)}=\mathcal{C}(C,\Sigma A)$, take a distinguished triangle

$$\xymatrix@C=2em{ A\ar[r]^f&B\ar[r]^g&C\ar[r]^{h\ \ \ \ }&\Sigma A}$$
and define as $\mathfrak{s}(\delta)=\xymatrix{[A\ar[r]^f&B\ar[r]^g&C].}$ With this definition, $(\mathcal{C}, \mathbb{E}, \mathfrak{s})$ becomes an extriangulated category (see \cite[Proposition 3.22(1)]{NP}).

\begin{thm}\label{cor:3.13} Let $(\mathcal{X},\mathcal{Y})$ be a {balanced pair} in a triangulated category $\mathcal{C}$. Then $\xi:=\xi_{\mathcal{X}}=\xi^{\mathcal{Y}}$ is a proper class in $\mathcal{C}$. With the notation above, we set $\mathbb{E}_\xi:=\mathbb{E}|_\xi$, that is, $$\mathbb{E}_\xi(C, A)=\{\delta\in\mathbb{E}(C, A)~|~\delta~ \textrm{is realized as an $\mathbb{E}$-triangle}\xymatrix{A\ar[r]^x&B\ar[r]^y&C\ar@{-->}[r]^{\delta}&}~\textrm{in}~\xi\}$$ for any $A, C\in\mathcal{C}$, and $\mathfrak{s}_\xi:=\mathfrak{s}|_{\mathbb{E}_\xi}$. Hence $(\mathcal{C}, \mathbb{E}_\xi, \mathfrak{s}_\xi)$ is an extriangulated category. Moreover, we have the following
\begin{enumerate}
\item[(1)] $(\mathcal{C}, \mathbb{E}_\xi, \mathfrak{s}_\xi)$ is a triangulated category if and only if $\mathcal{X}=\mathcal{Y}=0$; and

\item[(2)] $(\mathcal{C}, \mathbb{E}_\xi, \mathfrak{s}_\xi)$ is an exact category if and only if $\mathcal{X}=\mathcal{Y}=\mathcal{C}$.
\end{enumerate}
\end{thm}

To find more examples of extriangulated categories which are neither exact nor triangulated is an interested topic (see \cite{NP}, \cite{ZZ} and \cite{HZZ}). Theorem \ref{cor:3.13} provides a systematical way to produce ample examples of extriangulated categories which are neither exact nor triangulated (see Remark \ref{rem:3.14}, especially Example \ref{Ex:3.10}) which include \cite[Remark 3.3(1)]{HZZ} as a particular case (see Example \ref{Ex:3.8}).

The contents of this paper are outlined as follows. In Section \ref{preliminaries}, we fix notations and recall some definitions and basic facts used throughout the paper. In Section \ref{proof}, we first introduce and study the balanced pairs in a triangulated category $\mathcal{C}$, and then we give the proof of Theorem \ref{thm:3.11} and Theorem \ref{cor:3.13}.

\section{Preliminaries}\label{preliminaries}
Throughout the paper, we fix a triangulated category $\mathcal{C}=(\mathcal{C},\Sigma, \Delta)$, $\Sigma$ is the suspension
functor and $\Delta$ is the triangulation.

\begin{remark} {\rm Some equivalent formulations for the Octahedral axiom (Tr4), named base change
and cobase change, are given in \cite[2.1]{Bel1}, which are more convenient to use.}
\end{remark}

A triangle $\xymatrix@C=2em{(T): A\ar[r]^f&B\ar[r]^g&C\ar[r]^{h\ \ \ \ }&\Sigma A\in{\Delta}}$ is {\it split} if $h=0$.
It is easy to see that if $(T)$ is split, then the morphisms $f$, $g$ induce a direct sum decomposition $B\cong A\oplus C$.
 The full subcategory of $\Delta$ consisting of the split triangles will be denoted by $\Delta_0$.

  The following definitions are quoted from \cite[Section 2]{Bel1}. A class of triangles $\xi$ is {\it closed under base change} if for any triangle $\xymatrix@C=1.5em{A\ar[r]^f&B\ar[r]^g&C\ar[r]^{h\ \ \ \ }&\Sigma A}$ in $\xi$ and any morphism $\xymatrix@C=1.5em{\varepsilon: E\ar[r]& C,}$
  one gets from the commutative diagram of triangles $$\xymatrix{
  0 \ar[d] \ar[r] & M \ar[d]_{\alpha} \ar[r]^{=} & M\ar[d]_{\delta} \ar[r] & 0\ar[d] \\
  A \ar[d]_{||} \ar[r]^{f'} & G \ar[d]_{\beta} \ar[r]^{g'} & E\ar[d]_{\varepsilon} \ar[r]^{h'} & \Sigma A\ar[d]_{||} \\
  A \ar[d] \ar[r]^{f} & B\ar[d]_{\gamma} \ar[r]^{g} & C \ar[d]_{\zeta} \ar[r]^{h} & \Sigma A \ar[d] \\
  0 \ar[r]& \Sigma M\ar[r]^{=} & \Sigma M \ar[r] & 0   }
  $$
that the triangle $\xymatrix@C=1.5em{A\ar[r]^{f'}&G\ar[r]^{g'}&E\ar[r]^{h'}&\Sigma A}$ belongs to $\xi$. Dually, one has the notion that a class of triangles $\xi$ is {\it closed under cobase change}. A class of triangles $\xi$ is {\it closed under suspension} if for any triangle $\xymatrix@C=1.5em{A\ar[r]^f&B\ar[r]^g&C\ar[r]^{h\ \ \ \ }&\Sigma A}$ in $\xi$, the triangle $$\xymatrix@C=3em{\Sigma^iA\ar[r]^{(-1)^i\Sigma^if}&\Sigma^iB\ar[r]^{(-1)^i\Sigma^ig}&\Sigma^iC\ar[r]^{(-1)^i\Sigma^ih}&\Sigma^{i+1} A}$$ belongs to $\xi$  for all $i\in \mathbb{Z}$. A class of triangles $\xi$ is called {\it saturated} if in the situation of base change, whenever the third vertical and the second horizontal triangles are in $\xi$, then the  triangle $\xymatrix{A\ar[r]^f&B\ar[r]^g&C\ar[r]^{h }&\Sigma A}$  is in $\xi$.

  \begin{definition} {\rm (see \cite[Definition 2.2]{Bel1}) A full subcategory $\xi\subseteq \Delta$ is called a {\it proper class} of triangles if the following conditions hold:

  (1) $\xi$ is closed under isomorphisms, finite coproducts and $\Delta_0\subseteq \xi\subseteq \Delta$.

  (2) $\xi$ is closed under suspensions and is saturated.

  (3) $\xi$ is closed under base and cobase change.}

  \end{definition}
There are more interesting examples of proper classes of triangles enumerated in \cite[Example 2.3]{Bel1}. Throughout the paper we fix a proper class of triangles $\xi$ in the triangulated category of $\mathcal{C}$.

  \begin{definition} {\rm (see \cite[Definition 4.1]{Bel1}) An object $P\in\mathcal{C}$  is called  {\it $\xi$-projective}  if for any triangle $\xymatrix{A\ar[r]& B\ar[r]& C \ar[r]& \Sigma A}$ in $\xi$, the induced sequence of abelian groups $$\xymatrix{0\ar[r]& \mathcal{C}(P,A)\ar[r]& \mathcal{C}(P,B)\ar[r]&\mathcal{C}(P,C)\ar[r]& 0}$$ is exact. The triangulated category $\mathcal{C}$ is said to  have {\it  enough $\xi$-projectives}  provided that for each object $A$ there is a triangle $\xymatrix{K\ar[r]& P\ar[r]&A\ar[r]& \Sigma K}$  in $\xi$ with $P\in\mathcal{P}(\xi)$.

  Dually, one can define {\it $\xi$-injective} objects and {\it enough $\xi$-injectives}.}
  \end{definition}

 It is easy to check that the full subcategory $\mathcal{P}(\xi)$ of $\xi$-projective objects and the full subcategory $\mathcal{I}(\xi)$ of $\xi$-injective objects are full, additive, closed under isomorphisms, direct summands and {\it $\Sigma$-stable}, i.e. $\Sigma (\mathcal{P}(\xi))=\mathcal{P}(\xi)$ and $\Sigma (\mathcal{I}(\xi))=\mathcal{I}(\xi)$.

A {\it $\xi$-exact} complex $\mathbf{X}$ is a diagram $\xymatrix{\cdots\ar[r]&X_1\ar[r]^{d_1}&X_0\ar[r]^{d_0}&X_{-1}\ar[r]&\cdots}$ in $\mathcal{C}$ such that  there exists triangle $\xymatrix{K_{n+1}\ar[r]^{g_n}&X_n\ar[r]^{f_n}&K_n\ar[r]^{h_n\ \ \ }&\Sigma K_{n+1}}$ in $\xi$ for each integer $n$ and the differential is defined as $d_n=g_{n-1}f_n$ for each $n$.

If $\mathcal{C}$ has enough $\xi$-projectives, then for any object in $A\in{\mathcal{C}}$, there exists a $\xi$-exact complex $\xymatrix@C=1.5em{\mathbf{P}\ar[r]& A}$ such that $P_{i}\in\mathcal{P}(\xi)$ for all $i\geqslant0$, which is said to be a {\it $\xi$-projective resolution} of $A$ (see \cite[Definition 4.7]{Bel1}).

\begin{definition} {\rm (see \cite[Section 4]{Bel1})} {\rm For any objects $A, B$ of $\mathcal{C}$, choose a $\xi$-projective resolution $\xymatrix{\mathbf{P}\ar[r]& A}$ of  $A$. For any integer $n\geqslant 0$, the \emph{$\xi$-cohomology groups} are defined as
$\xi{\rm xt}_{\mathcal{P}}^n(A,B)=H^n(\mathcal{C}(\mathbf{P},B))$.}
\end{definition}

The {\it $\xi$-projective dimension} $\xi$-${\rm pd} A$ of $A\in\mathcal{C}$ is defined inductively.
 If $A\in\mathcal{P}(\xi)$, then define $\xi$-${\rm pd} A=0$.
Next if $\xi$-${\rm pd} A>0$, define $\xi$-${\rm pd} A\leqslant n$ if there exists a triangle $$\xymatrix{K\ar[r]& P\ar[r]&A\ar[r]& \Sigma K}$$  in $\xi$ with $P\in \mathcal{P}(\xi)$ and $\xi$-${\rm pd} K\leqslant n-1$.
Finally we define $\xi$-${\rm pd} A=n$ if $\xi$-${\rm pd} A\leqslant n$ and $\xi$-${\rm pd} A\nleq n-1$. Of course we set $\xi$-${\rm pd} A=\infty$, if $\xi$-${\rm pd} A\neq n$ for all $n\geqslant 0$.

Dually, we can define the {\it $\xi$-injective dimension}  $\xi$-${\rm id} A$ of $A\in\mathcal{C}$.

\begin{definition}{\rm (see \cite[Definition 3.2]{AS2})} {\rm A triangle $\xymatrix{A\ar[r]& B\ar[r]& C\ar[r]& \Sigma A}$ in $\xi$ is called  {\it $\mathcal{C}(-,\mathcal{P}(\xi))$-exact} if for any $P\in\mathcal{P}(\xi)$, the induced sequence of abelian groups
{$$\xymatrix{0\ar[r]&\mathcal{C}(C,P)\ar[r]& \mathcal{C}(B,P)\ar[r]&\mathcal{C}(A,P)\ar[r]& 0}$$}
is exact. A $\xi$-exact complex  $$\xymatrix{\cdots\ar[r]&X_1\ar[r]^{d_1}&X_0\ar[r]^{d_0}&X_{-1}\ar[r]&\cdots}$$ in $\mathcal{C}$  is called {\it $\mathcal{C}(-,\mathcal{P}(\xi))$-exact} if for any integer $n$, there exists a   $\mathcal{C}(-,\mathcal{P}(\xi))$-exact triangle in $\xi$ $\xymatrix@C=1.5em{K_{n+1}\ar[r]^{g_n}&X_n\ar[r]^{f_n}&K_n\ar[r]^{h_n}&\Sigma K_{n+1} }$  and the differential is defined as $d_n=g_{n-1}f_n$ for each $n$.
}
\end{definition}

\begin{definition} {\rm (see \cite[Definition 3.6]{AS1}) A  {\it complete $\xi$-projective resolution}  is a $\xi$-exact complex $\xymatrix@C=2em{\mathbf{P}:\cdots\ar[r]&P_1\ar[r]^{d_1}&P_0\ar[r]^{d_0}&P_{-1}\ar[r]&\cdots}$ in $\mathcal{C}$ such that $\mathbf{P}$ is
$\mathcal{C}(-,\mathcal{P}(\xi))$-exact and $P_n$ is $\xi$-projective for each integer $n$. Let $\mathbf{P}$ be a $\xi$-exact complex. So for any integer $n$, there exists  a triangle $\xymatrix{K_{n+1}\ar[r]^{g_n}& P_n\ar[r]^{f_n}&K_n\ar[r]^{h_n}& \Sigma K_{n+1}}$  in $\xi$,  the object $K_n$ is called \emph{$\xi$-$\mathcal{G}$projective}.

Dually, one can define {\it complete $\xi$-injective coresolution} and {\it $\xi$-$\mathcal{G}$injective object}.}
\end{definition}

Similar to the way of defining $\xi$-projective and $\xi$-injective dimensions, for an object $A\in{\mathcal{C}}$, the $\xi$-Gprojective dimension $\xi$-${\rm \mathcal{G}pd} A$ and $\xi$-Ginjective dimension $\xi$-${\rm \mathcal{G}id} A$ are defined inductively in \cite{AS1}.

Throughout this paper, the full subcategory of $\xi$-projective (respectively, $\xi$-injective) objects is denoted by $\mathcal{P(\xi)}$ (respectively, $\mathcal{I(\xi)}$). We denote by $\mathcal{GP}(\xi)$ (respectively, $\mathcal{GI}(\xi)$) the class of $\xi$-$\mathcal{G}$projective (respectively, $\xi$-$\mathcal{G}$injective) objects.
It is obvious that $\mathcal{P(\xi)}$ $\subseteq$ $\mathcal{GP}(\xi)$ and $\mathcal{I(\xi)}$ $\subseteq$ $\mathcal{GI}(\xi)$.

\section{Proofs of the results}\label{proof}
Assume that $\mathcal{C}=(\mathcal{C},\Sigma, \Delta)$ is a triangulated category with $\Sigma$ the suspension
functor and $\Delta$ the triangulation.

\begin{definition} {\rm Let $(T):\xymatrix@C=1.5em{A\ar[r]^{f}&B\ar[r]^{g}&C\ar[r]^{h}&\Sigma A}$ be a triangle in $\mathcal{C}$ and $\mathcal{X}$ a full additive subcategory of $\mathcal{C}$.
\begin{enumerate}
\item
The triangle $(T)$ is said to be \emph{left $\mathcal{C}(\mathcal{X},-)$-exact} if the induced map $$\mathcal{C}(X,f):\mathcal{C}(X,A)\rightarrow \mathcal{C}(X,B)$$
is a monomorphism for any object $X\in{\mathcal{X}}$.

\item The triangle $(T)$ is said to be \emph{right $\mathcal{C}(\mathcal{X},-)$-exact} if the induced map $$\mathcal{C}(X,g):\mathcal{C}(X,B)\rightarrow \mathcal{C}(X,C)$$
is an epimorphism for any object $X\in{\mathcal{X}}$.

\item The triangle $(T)$ is said to be \emph{$\mathcal{C}(\mathcal{X},-)$-exact} if it is both left and right $\mathcal{C}(\mathcal{X},-)$-exact.
\end{enumerate}
Dually, we have the notions of left $\mathcal{C}(-,\mathcal{X})$-exact, right $\mathcal{C}(-,\mathcal{X})$-exact and $\mathcal{C}(-,\mathcal{X})$-exact.}
\end{definition}

Recall that a subcategory $\mathcal{X}$ is said to be \emph{$\Sigma$-stable} if $\Sigma(\mathcal{X})=\mathcal{X}$. We have

\begin{lem}\label{lem1} Let $\mathcal{X}$ be a $\Sigma$-stable full additive subcategory of $\mathcal{C}$. Then the following are equivalent for any triangle $(T):\xymatrix@C=1.5em{A\ar[r]^{f}&B\ar[r]^{g}&C\ar[r]^{h}&\Sigma A}.$
\begin{enumerate}
\item The triangle $(T)$ is left $\mathcal{C}(\mathcal{X},-)$-exact.
\item The triangle $(T)$ is right $\mathcal{C}(\mathcal{X},-)$-exact.
\item The triangle $(T)$ is $\mathcal{C}(\mathcal{X},-)$-exact.
\end{enumerate}
\end{lem}
\begin{proof} It follows from \cite[Lemma 2.6]{HYF}.
\end{proof}

 In the following, we set $\xi_{\mathcal{X}}$ be the class of triangles satisfying the condition of Lemma \ref{lem1}.
 Let $\mathcal{Y}$ be a full additive subcategory of $\mathcal{C}$. Recall that a morphism $f:C\rightarrow Y$ with $Y\in{\mathcal{}}$ is called a \emph{left $\mathcal{Y}$-approximation}
(or a \emph{$\mathcal{Y}$-preenvelope}) of $C$ if $\mathcal{C}(f,Y'):\mathcal{C}(Y,Y')\rightarrow \mathcal{C}(C,Y')$ is surjective for any object $Y'\in{\mathcal{Y}}$. If any $C\in \mathcal{C}$ admits a  left $\mathcal{Y}$-approximation, then $\mathcal{Y}$ is \emph{covariantly finite} in $\mathcal{C}$. Dually, one has the notions of \emph{right $\mathcal{Y}$-approximation} (or a \emph{$\mathcal{Y}$-precover}) and \emph{contravariantly finite subcategory} in $\mathcal{C}$.

\begin{prop}\label{prop1} Let $\mathcal{X}$ be a $\Sigma$-stable full additive subcategory of $\mathcal{C}$ which is closed under direct summands.  Then $\xi_{\mathcal{X}}$ is a proper class of $\mathcal{C}$. Moreover, $\mathcal{X}$ is a contravariantly finite in $\mathcal{C}$ if and only if $\mathcal{C}$ has enough $\xi_{\mathcal{X}}$-projectives and $\mathcal{X}=\mathcal{P}(\xi_{\mathcal{X}})$.
\end{prop}
\begin{proof}It is easy to see that $\xi_{\mathcal{X}}$ is closed under isomorphisms, finite coproducts and containing all split triangles. Next we claim that
$\xi_{\mathcal{X}}$ is closed under base change and cobase change. Consider the following commutative diagram of triangles
{$$\xymatrix{
  0 \ar[d] \ar[r] & M \ar[d]_{\alpha} \ar[r]^{=} & M\ar[d]_{\lambda} \ar[r] & 0\ar[d] &  \\
  A \ar[d]_{||} \ar[r]^{f'} & G \ar[d]_{\beta} \ar[r]^{g'} & E\ar[d]_{\varepsilon} \ar[r]^{h'} & \Sigma A\ar[d]_{||} &(\ast)  \\
  A \ar[d] \ar[r]^{f} & B\ar[d]_{\gamma} \ar[r]^{g} & C \ar[d]_{\zeta} \ar[r]^{h} & \Sigma A \ar[d]& \\
  0 \ar[r]& \Sigma M\ar[r]^{=} & \Sigma M \ar[r] & 0. &  }
  $$}
If the triangle $\xymatrix@C=1.5em{A\ar[r]^{f}&B\ar[r]^{g}&C\ar[r]^{h}&\Sigma A}$ belongs to $\xi_{\mathcal{X}}$, then
  $\mathcal{C}(\mathcal{X},f)= \mathcal{C}(\mathcal{X},\beta)\mathcal{C}(\mathcal{X},f')$ is monic. Thus $\mathcal{C}(\mathcal{X},f')$ is monic, and hence $\xi_{\mathcal{X}}$ is closed under base change. Similarly, we can prove that $\xi_{\mathcal{X}}$ is closed under cobase change. To prove that $\xi_{\mathcal{X}}$ is closed under saturated,  we assume that two triangles $\xymatrix@C=1.5em{A\ar[r]^{f'}&G\ar[r]^{g'}&E\ar[r]^{h'}&\Sigma A}$ and $\xymatrix@C=1.5em{M\ar[r]^{\lambda}&E\ar[r]^{\varepsilon}&C\ar[r]^{\zeta}&\Sigma M}$ in the diagram ($\ast$) belong to $\xi_{\mathcal{X}}$. Then $\mathcal{C}(\mathcal{X},\varepsilon)$ and $\mathcal{C}(\mathcal{X},g')$ are epic. Thus $\mathcal{C}(\mathcal{X},g)$$\mathcal{C}(\mathcal{X},\beta)=$ $\mathcal{C}(\mathcal{X},\varepsilon)\mathcal{C}(\mathcal{X},g')$ is epic, and so is $\mathcal{C}(\mathcal{X},g)$, as desired.

Finally, if $\mathcal{C}$ has enough $\xi_{\mathcal{X}}$-projectives and $\mathcal{X}=\mathcal{P}(\xi_{\mathcal{X}})$, then it is easy to check that $\mathcal{X}$ is a contravariantly finite in $\mathcal{C}$. Conversely, assume that $\mathcal{X}$ is a contravariantly finite in $\mathcal{C}$. It is clear that $\mathcal{X}\subseteq \mathcal{P}(\xi_{\mathcal{X}})$. For the reverse containment, let $P$ be an object in $\mathcal{P}(\xi_{\mathcal{X}})$. Then there exists a triangle $\xymatrix@C=1.5em{K\ar[r]&X\ar[r]&P\ar[r]&\Sigma K}$ in $\xi_{\mathcal{X}}$ with $X\in{\mathcal{X}}$. Thus this triangle is split, and hence $P\in \mathcal{X}$. This completes the proof.
\end{proof}

The next two results are dual to Lemma \ref{lem1} and Proposition \ref{prop1}, we omit the proof.

\begin{lem}\label{lem2} Let $\mathcal{Y}$ be a $\Sigma$-stable full additive subcategory of $\mathcal{C}$. Then the following are equivalent for any triangle $(T):\xymatrix@C=1.5em{A\ar[r]^{f}&B\ar[r]^{g}&C\ar[r]^{h}&\Sigma A}$.
\begin{enumerate}
\item The triangle $(T)$ is left $\mathcal{C}(-,\mathcal{Y})$-exact.
\item The triangle $(T)$ is right $\mathcal{C}(-,\mathcal{Y})$-exact.
\item The triangle $(T)$ is $\mathcal{C}(-,\mathcal{Y})$-exact.
\end{enumerate}
\end{lem}

We set $\xi^{\mathcal{Y}}$ be the class of triangles satisfying the condition of Lemma \ref{lem2}.

\begin{prop}\label{prop2} Let $\mathcal{Y}$ be a $\Sigma$-stable full additive subcategory of $\mathcal{C}$ which is closed under direct summands. Then $\xi^{\mathcal{Y}}$ is a proper class of $\mathcal{C}$. Moreover, $\mathcal{Y}$ is a covariantly finite in $\mathcal{C}$ if and only if $\mathcal{C}$ has enough $\xi^{\mathcal{Y}}$-injectives and $\mathcal{Y}=\mathcal{I}(\xi^{\mathcal{Y}})$.
\end{prop}

\begin{definition} {\rm Let $M$ be an object in $\mathcal{C}$, and let $\mathcal{X}$ and $\mathcal{Y}$ be $\Sigma$-stable full additive subcategories of $\mathcal{C}$.
\begin{enumerate}

\item An {\it $\mathcal{X}$-resolution} of $M$ is a diagram $X^{\bullet}\rightarrow M$ such that $X^{\bullet}:=\cdots \rightarrow X_{1}\rightarrow X_{0}\rightarrow0$ is a complex with $X_i\in{\mathcal{X}}$ for all $i\geq 0$, and $\cdots \rightarrow X_{1}\rightarrow X_{0}\rightarrow M$ is a $\xi_{\mathcal{X}}$-exact complex. Moreover, the $\mathcal{X}$-resolution $X^{\bullet}\rightarrow M$ of $M$ is called \emph{$\mathcal{C}(-,\mathcal{Y})$-exact} if its $\xi_{\mathcal{X}}$-exact complex is $\mathcal{C}(-,\mathcal{Y})$-exact.

 \item   A {\it $\mathcal{Y}$-coresolution} of $M$ is a diagram $M\rightarrow Y^{\bullet}$ such that $Y^{\bullet}:=0\rightarrow Y_{0}\rightarrow Y_{-1}\rightarrow \cdots$ is a complex with $Y_i\in{\mathcal{Y}}$ for all $i\leq 0$, and $M\rightarrow Y_{0}\rightarrow Y_{-1}\rightarrow \cdots$ is a $\xi^{\mathcal{Y}}$-exact complex. Moreover, the $\mathcal{Y}$-coresolution $M\rightarrow Y^{\bullet}$ of $M$ is called \emph{$\mathcal{C}(\mathcal{X},-)$-exact} if its $\xi^{\mathcal{Y}}$-exact complex is $\mathcal{C}(\mathcal{X},-)$-exact.
\end{enumerate}}
\end{definition}

The next result characterizes when $\mathcal{C}$ has enough $\xi_{\mathcal{X}}$-projectives and $\xi^{\mathcal{Y}}$-injectives.

\begin{prop}\label{thm1} Assume that $\mathcal{X}$ and $\mathcal{Y}$ are $\Sigma$-stable full additive subcategories of $\mathcal{C}$ which are closed under direct summands. Then the following are equivalent.
\begin{enumerate}
\item $\xi_{\mathcal{X}}$=$\xi^{\mathcal{Y}}$, $\mathcal{X}=\mathcal{P}(\xi_{\mathcal{X}})$, $\mathcal{Y}=\mathcal{I}(\xi^{\mathcal{Y}})$ and every object in $\mathcal{C}$ has enough $\xi_{\mathcal{X}}$-projectives and enough $\xi^{\mathcal{Y}}$-injectives.

\item The pair {\rm($\mathcal{X},\mathcal{Y}$)} satisfies:
\begin{enumerate}
\item $\mathcal{X}$ is contravariantly finite and $\mathcal{Y}$ is covariantly finite in $\mathcal{C}$.

\item For any object $M\in{\mathcal{C}}$, there is an $\mathcal{X}$-resolution $X^{\bullet}\rightarrow M$ such that it is $\mathcal{C}(-,\mathcal{Y})$-exact.
 \item For any object $N\in{\mathcal{C}}$, there is a $\mathcal{Y}$-coresolution $N\rightarrow Y^{\bullet}$ such that it is  $\mathcal{C}(\mathcal{X},-)$-exact.
\end{enumerate}
\end{enumerate}
\end{prop}
\begin{proof} $(1)\Rightarrow(2)$ follows from Propositions \ref{prop1} and \ref{prop2}.

$(2)\Rightarrow(1)$. By Lemmas \ref{prop1} and \ref{prop2}, it suffices to show $\xi_{\mathcal{X}}=\xi^{\mathcal{Y}}$. Let  $\xymatrix@C=1.5em{A\ar[r]^{f}&B\ar[r]^{g}&C\ar[r]^{h}&\Sigma A}$ be a triangle in $\xi_{\mathcal{X}}$. By hypothesis, there is an $\mathcal{X}$-resolution $X^{\bullet}\rightarrow C$ of $C$ such that it is $\mathcal{C}(-,\mathcal{Y})$-exact. Then there exists a $\xi_{\mathcal{X}}$-exact complex $$\xymatrix{\cdots\ar[r]&X_1\ar[r]^{d_1}&X_0\ar[r]^{d_0}&C}$$ in $\mathcal{C}$ which is $\mathcal{C}(-,\mathcal{Y})$-exact. This gives us a
triangle $\xymatrix{K_{1}\ar[r]^{g_0}&X_0\ar[r]^{f_0}&C\ar[r]^{h_0\ \ \ }&\Sigma K_{1}}$ which is $\mathcal{C}(-,\mathcal{Y})$-exact, and hence we have the following commutative diagram of triangles

$$\xymatrix{K_{1}\ar[r]^{g_{0}}\ar@{.>}[d]^{\alpha}&X_{0}\ar[r]^{f_{0}}\ar@{.>}[d]^{\beta}&C\ar[r]^{h_{0}}\ar@{=}[d]&\Sigma K_{1}\ar@{.>}[d]\\
A\ar[r]^{f}&B\ar[r]^{g}&C\ar[r]^{h}&\Sigma A.}$$
Let $Y$ be an object in $\mathcal{Y}$. Applying $\mathcal{C}(-,Y)$ to the commutative diagram above, we have the following commutative diagram

$$\xymatrix{&\mathcal{C}(C,Y)\ar[r]^{\mathcal{C}(g,Y)}\ar@{=}[d]&\mathcal{C}(B,Y)\ar[r]^{\mathcal{C}(f,Y)}\ar[d]^{\mathcal{C}(\beta,Y)}
&\mathcal{C}(A,Y)\ar[d]^{\mathcal{C}(\alpha,Y)}&\\
0\ar[r]&\mathcal{C}(C,Y)\ar[r]^{\mathcal{C}(f_0,Y)}&\mathcal{C}(X_{0},Y)\ar[r]^{\mathcal{C}(g_0,Y)}&\mathcal{C}(K_1,Y)\ar[r]&0.}$$
Note that $\mathcal{C}(f_{0},Y):\mathcal{C}(C,Y)\rightarrow\mathcal{C}(X_{0},Y)$ is monic. It follows that $\mathcal{C}(g,Y):\mathcal{C}(C,Y)\rightarrow\mathcal{C}(B,Y)$ is monic. So $\xymatrix@C=1.5em{A\ar[r]^{f}&B\ar[r]^{g}&C\ar[r]^{h}&\Sigma A}$ is $\mathcal{C}(-,\mathcal{Y})$-exact by Lemma \ref{lem2} and it belongs to $\xi^{\mathcal{Y}}$.
This implies that $\xi_{\mathcal{X}}\subseteq\xi^{\mathcal{Y}}$. Dually, one can show that $\xi^{\mathcal{Y}}\subseteq\xi_{\mathcal{X}}$.
\end{proof}

Now we introduce the notion of balanced pairs, which parallels \cite[Definition 1.1]{chen}.

\begin{definition}\label{df:3.7} {\rm Assume that $\mathcal{X}$ and $\mathcal{Y}$ are $\Sigma$-stable full additive subcategories of $\mathcal{C}$ which are closed under direct summands. The pair $(\mathcal{X},\mathcal{Y})$ is called a \emph{balanced pair} if it satisfies the equivalent conditions of Proposition \ref{thm1}.}
\end{definition}

The following example comes from Krause \cite{Krause} and Beligiannis \cite{Bel1}.

\begin{Ex}\label{Ex:3.8} Assume that $\mathcal{C}$ is a compactly generated triangulated category. Then the class $\xi$ of pure triangles
(which is induced by the compact objects) is proper and $\mathcal{C}$ has enough $\xi$-projectives and $\xi$-injectives. Denote by $\mathcal{PP}$ the class of pure projective objects and by $\mathcal{PI}$ the class of pure injective objects in $\mathcal{C}$. It follows that $(\mathcal{PP},\mathcal{PI})$ is a balanced pair in $\mathcal{C}$.
\end{Ex}

Recall that $\mathcal{C}$ is called a \emph{$\xi$-Gorenstein triangulated category} \cite[Definition 4.6]{AS2}, if any object of $\mathcal{C}$ has both $\xi$-$\mathcal{G}$projective and $\xi$-$\mathcal{G}$injective dimension less than or equal to a nonnegative integer $n$.

\begin{Ex}\label{Ex:3.9} Assume that $\mathcal{C}$ is a $\xi$-Gorenstein triangulated category. By \cite[Remark 4.4]{AS1} and its dual, one has that the class $\mathcal{GP}(\xi)$ of $\xi$-$\mathcal{G}$projective objects is contravariantly finite in $\mathcal{C}$ and the class $\mathcal{GI}(\xi)$ of $\xi$-$\mathcal{G}$injective objects is convariantly finite in $\mathcal{C}$. So $(\mathcal{GP}(\xi),\mathcal{GI}(\xi))$ is a balanced pair by \cite[Theorem4.3]{RL1}.
\end{Ex}

Let $R$ be an associative ring with identity. We denote the category of left $R$-modules by $\textrm{Mod}R$. Assume that $(\mathscr{X},\mathscr{Y})$ is a cotorsion pair in $\textrm{Mod}R$, that is $\mathscr{Y}=\{M:\mbox{Ext}^1_R(X,M)=0\mbox{ for every } X\in \mathscr{X}\}$ and $\mathscr{X}=\{N:\mbox{Ext}^1_R(N,Y)=0\mbox{ for every } Y\in\mathscr{Y}\}$. We introduce the following classes, which comes from \cite{G}:

\begin{enumerate}
\item $\widetilde{\mathscr{X}}$ is the class of all exact complexes $X^\bullet$ of left $R$-modules with each cycle $\textrm{Z}_{n}(X^\bullet)\in{\mathscr{X}}$. Similarly, one has the notion of $\widetilde{\mathscr{Y}}$.

\item $dg\mathscr{X}$ is the class of all complexes $X^\bullet$ of left $R$-modules satisfying that $X^\bullet_n \in {\mathscr{X}}$ for any $n\in{\mathbb{Z}}$
and every chain map $f^\bullet : X^\bullet \rightarrow Y^\bullet$ is null homotopic whenever $Y^\bullet\in{\widetilde{\mathscr{Y}}}$ . Similarly, $dg\mathscr{Y}$ can be also defined.
\end{enumerate}

\begin{Ex}\label{Ex:3.10} Assume that ${\rm K}(R)$ is the  homotopy
category of left $R$-modules. Denote by $\mathscr{P}$ the class of projective left $R$-modules and by $\mathscr{I}$ the class of injective left $R$-modules. Then $(dg\mathscr{P}, dg\mathscr{I})$ is a balanced pair in $\textrm{K}(R)$.
\end{Ex}
\begin{proof} It is easy to check that $dg\mathscr{P}$ and $dg\mathscr{I}$ are $\Sigma$-stable full additive subcategories of ${\rm K}(R)$. It follows from \cite[Lemmas 4.2 and 4.5]{Chenwj} that $dg\mathscr{P}$ is contravariantly finite in $\textrm{K}(R)$ and $dg\mathscr{I}$ is covariantly finite in $\textrm{K}(R)$. Note that every object $M$ in ${\rm K}(R)$ has a triangle $$\xymatrix@C=1.5em{K\ar[r]^{f}&X\ar[r]^{g}&M\ar[r]^{h}&\Sigma K}$$ with $X\in{dg\mathscr{P}}$ and $K$ an exact complex by  \cite[Lemma 4.5]{Chenwj}. Then one can construct a $dg\mathscr{P}$-resolution $X^{\bullet}\rightarrow M$. To prove that it is $\textrm{K}(R)(-,dg\mathscr{I})$-exact, it suffices to show that any triangle $(T):\xymatrix@C=1.5em{K\ar[r]^{f}&X\ar[r]^{g}&M\ar[r]^{h}&\Sigma K}$ with $X\in{dg\mathscr{P}}$ and $K$ an exact complex is $\textrm{K}(R)(-,dg\mathscr{I})$-exact. Let $Y$ be an object in $dg\mathscr{I}$. Then we have the following exact sequence
$$\xymatrix@C=1.5em{\textrm{K}(R)(M,Y)\ar[r]&\textrm{K}(R)(X,Y)\ar[r]&\textrm{K}(R)(K,Y).}$$
Since $\textrm{K}(R)(K,Y)=0$, it follows that the triangle $(T)$ is $\textrm{K}(R)(-,dg\mathscr{I})$-exact by Lemma \ref{lem2}, as desired. Similarly, one can  construct a $dg\mathscr{I}$-coresolution $M\rightarrow Y^{\bullet}$ which is $\textrm{K}(R)(dg\mathscr{P},-)$-exact. This completes the proof.
\end{proof}

We are now in a position to prove the main results of this paper.
\begin{para}{\bf Proof of Theorem \ref{thm:3.11}.} Let $(\mathcal{X},\mathcal{Y})$ be a balanced pair in $\mathcal{C}$. Then $\xi_{\mathcal{X}}=\xi^{\mathcal{Y}}$ is the desired proper class such that $\mathcal{X}=\mathcal{P}(\xi_{\mathcal{X}})$ and $\mathcal{Y}=\mathcal{I}(\xi^{\mathcal{Y}})$ by Proposition \ref{thm1}. Conversely, assume that $\xi$ is a proper class in $\mathcal{C}$ with enough $\xi$-projectives and enough $\xi$-injectives. We put $(\mathcal{X},\mathcal{Y})=(\mathcal{P}(\xi),\mathcal{I}(\xi))$. Note that $\Sigma (\mathcal{P}(\xi))=\mathcal{P}(\xi)$ and $\Sigma (\mathcal{I}(\xi))=\mathcal{I}(\xi)$. Then $(\mathcal{X},\mathcal{Y})=(\mathcal{P}(\xi),\mathcal{I}(\xi))$ is a balanced pair by Proposition \ref{thm1}.

For any balanced pair $(\mathcal{X},\mathcal{Y})$, one can check that $\Phi\Psi(\mathcal{X},\mathcal{Y})=\Phi(\xi_{\mathcal{X}}=\xi^{\mathcal{Y}})=
(\mathcal{P}(\xi_{\mathcal{X}}),\mathcal{I}(\xi^{\mathcal{Y}}))=(\mathcal{X},\mathcal{Y})$. On the other hand, assume that $\xi$ is a proper class  in $\mathcal{C}$ with enough $\xi$-projectives and enough $\xi$-injectives, it is straightforward to see that $\Psi\Phi(\xi)=\Psi((\mathcal{P}(\xi),\mathcal{I}(\xi)))=
\Psi(\mathcal{P}(\xi_{\mathcal{X}}),\mathcal{I}(\xi^{\mathcal{Y}}))=\xi$. This completes the proof.\hfill$\Box$
\end{para}

\begin{para}{\bf Proof of Theorem \ref{cor:3.13}.} By Theorem \ref{thm:3.11}, we have that $\xi$ is a proper class in $\mathcal{C}$ with enough $\xi$-projectives and enough $\xi$-injectives. If follows from \cite[Theorem 3.2]{HZZ} that $(\mathcal{C}, \mathbb{E}_\xi, \mathfrak{s}_\xi)$ is an extriangulated category. Moreover,

(1) If $\mathcal{X}=\mathcal{Y}=0$, then $\xi$ is the class of all triangles by Theorem \ref{thm:3.11}. Hence $(\mathcal{C}, \mathbb{E}_\xi, \mathfrak{s}_\xi)$ is a triangulated category. Conversely, we assume that $(\mathcal{C}, \mathbb{E}_\xi, \mathfrak{s}_\xi)$ is a triangulated category. Then
for any object $X$ in $\mathcal{X}$ and any morphism $f:C\to X$, one has a triangle $$
(T1):\xymatrix@C=2em{ K\ar[r]&C\ar[r]&X\ar[r]&\Sigma K}$$
in $\xi$. Thus the triangle $(T1)$ is split, and hence $X=0$ and $\mathcal{X}=0$. Similarly, one can prove that $\mathcal{Y}=0$.

(2) If $\mathcal{X}=\mathcal{Y}=\mathcal{C}$, then $\xi$ is the class of split triangles by Theorem \ref{thm:3.11}. So $(\mathcal{C}, \mathbb{E}_\xi, \mathfrak{s}_\xi)$ is an exact category by \cite[Corollary 3.18]{NP}. Conversely, we assume that $(\mathcal{C}, \mathbb{E}_\xi, \mathfrak{s}_\xi)$ is an exact category. Then any triangle $(T2):\xymatrix@C=2em{ A\ar[r]^f&B\ar[r]^g&C\ar[r]^{h\ \ \ \ }&\Sigma A}$ in $\xi$ satisfies that $f$ is monomorphic. Thus the triangle $(T2)$ is split, and hence $\xi$ is the class of split triangles. So $\mathcal{X}=\mathcal{Y}=\mathcal{C}$. This completes the proof.\hfill$\Box$
\end{para}

\begin{rem}\label{rem:3.14} By Theorem \ref{cor:3.13}, one can get that  the balanced pairs constructed in Examples \ref{Ex:3.8}-\ref{Ex:3.10} can induce extriangulated categories which are neither exact nor triangulated
\end{rem}

\end{document}